\documentclass{amsart}
\usepackage{a4wide}
\usepackage{amssymb}

\newcommand{\IR}{\mathbb R}

\newcommand{\IN}{\mathbb N}

\newcommand{\br}{\mathbf r}

\newcommand{\e}{\varepsilon}

\newtheorem{theorem}{Theorem}[section]
\newtheorem{corollary}[theorem]{Corollary}
\newtheorem{lemma}[theorem]{Lemma}
\newtheorem{claim}[theorem]{Claim}
\newtheorem{problem}[theorem]{Problem}

\theoremstyle{definition}
\newtheorem{definition}[theorem]{Definition}

\title[Every $2$-dimensional Banach space has the Mazur--Ulam property]{Every $2$-dimensional Banach space has\\ the Mazur--Ulam property}
\author{Taras Banakh}
\address{Ivan Franko National University of Lviv (Ukraine) and Jan Kochanowski University in Kielce (Poland)}
\email{t.o.banakh@gmail.com}

\begin{document}
\begin{abstract} We prove that every isometry between the unit spheres of 2-dimensional Banach spaces extends to a linear isometry of the Banach spaces. This resolves the famous Tingley's problem in the class of 2-dimensional Banach spaces.
\end{abstract}
\maketitle

\section{Introduction}

By the classical result of Mazur and Ulam \cite{MU}, every bijective isometry between Banach spaces is affine. This result essentially asserts that the metric structure of a Banach space determines its linear structure. In  \cite{Man} Mankiewicz proved that every bijective isometry $f:B_X\to B_Y$ between the unit balls of two Banach spaces $X,Y$ extends to a linear isometry of the Banach spaces. In \cite{Tingley} Tingley asked if the unit balls in this result of Mankiewicz can be replaced by the unit spheres. More precisely, he posed the following problem (that remains unsolved more than thirty years).

\begin{problem}[Tingley, 1987]\label{prob:Tingley} Let $f:S_X\to S_Y$ be a bijective isometry of the unit spheres of two Banach spaces $X,Y$. Can $f$ be extended to a linear isometry between the Banach spaces $X,Y$?
\end{problem}

Here for a Banach space $(X,\|\cdot\|)$ by 
$$B_X=\{x\in X:\|x\|\le 1\}\quad\mbox{and}\quad S_X=\{x\in X:\|x\|=1\}$$we denote the unit ball and unit sphere of $X$, respectively.

Tingley's Problem~\ref{prob:Tingley} can be equivalently reformulated in terms of the Mazur--Ulam property,  introduced by Cheng and Dong \cite{CD} and widely used in the the literature, see e.g. \cite{BG}, \cite{CAP}, \cite{CAP2}, \cite{JVMCPR}, \cite{Li}, \cite{MO}, \cite{WX}.

\begin{definition} A Banach space $X$ is defined to have the {\em Mazur--Ulam property} if every  isometry $f:S_X\to S_Y$ of $S_X$ onto the unit sphere $S_Y$ of an arbitrary Banach space $Y$ extends to a linear isometry of the Banach spaces $X,Y$.
\end{definition} 

In fact, Tingley's Problem~\ref{prob:Tingley} asks whether every Banach space has the Mazur--Ulam property. Many classical Banach spaces (including $C(K)$, $c_0(\Gamma)$, $\ell_p(\Gamma)$, $L_p(\mu)$) have the Mazur--Ulam property, see \cite{Ding}, \cite{DL}. By the result of Kadets and Mart\'\i n \cite{KM}, every polyhedral finite-dimensional Banach space has the Mazur--Ulam property. 

The main result of this paper is the following theorem that answers Tingley's problem in the class of $2$-dimensional Banach spaces.

\begin{theorem}\label{t:main} Every $2$-dimensional Banach space has the Mazur--Ulam property.
\end{theorem}

Theorem~\ref{t:main} is a corollary of four partial answers to Tingley's problem. The first of them  was proved by Cabello S\'anchez in \cite{San}.

\begin{theorem}[Cabello S\'anchez]\label{t:CS} A $2$-dimensional Banach space has the Mazur--Ulam property if it is not strictly convex.
\end{theorem}

A Banach space is called 
\begin{itemize}
\item {\em strictly convex} if any convex subset of its unit sphere contains at most one point;
\item {\em smooth} if its unit ball has a unique supporting hyperplane at each point of the unit sphere.
\end{itemize}

The second ingredient of the proof of Theorem~\ref{t:main} was proved by Banakh and Cabello Sanchez in \cite{BCS}.

\begin{theorem}[Banakh, Cabello S\'anchez]\label{t:BCS} A $2$-dimensional Banach space has the Mazur--Ulam property if it is not smooth.
\end{theorem}

The third crucial ingredient of the proof of Theorem~\ref{t:main} concerns absolutely smooth Banach spaces. A $2$-dimensional Banach space $X$ is called {\em absolutely smooth} if there exists a differentiable map $\br:\IR\to S_X$ such that $\|\br'(s)\|=1$ for all $s\in\IR$ and the derivative $\br':\IR\to S_X$ is locally absolutely continuous. For absolutely smooth $2$-dimensional Banach spaces, Tingley's problem was answered in \cite{Ban} as follows.

\begin{theorem}[Banakh] Any isometry between the unit spheres of two absolutely smooth $2$-dimensional Banach spaces extends to a linear isometry of the Banach spaces.
\end{theorem}

Therefore, to derive Theorem~\ref{t:main} it remains to prove 

\begin{theorem}\label{t:main2} A $2$-dimensional Banach space has the Mazur--Ulam property if it is strictly convex and smooth but not absolutely smooth.
\end{theorem}

This theorem will be proved in Section~\ref{s:main} after some preparatory work, made in Sections~\ref{s:prep} and \ref{s:dif}.

\section{Preliminaries}\label{s:prep}

In this section we collect some definitions and known results that will be used in the proof our main result.  

\subsection{Smoothness properties of real functions}

Let $U$ be an open subset of the real line and $s\in U$. A function $f:U\to X$ to a Banach space $X$ is defined to be 
\begin{itemize}
\item {\em Lipschitz} at $s$ if there exists a constant $C$ such that\newline $\|f(s+\e)-f(s)\|\le C\cdot|\e|+o(\e)$ for a small $\e$;
\item {\em differentiable} at $s$ if there is a vector $f'(s)\in X$ such that\newline $f(s+\e)=f(s)+f'(s)\cdot\e+o(\e)$ for a small $\e$;
\item {\em twice differentiable} at $s$ if there are vectors $f'(s),f''(s)\in X$ such that\newline $f(s+\e)=f(s)+f'(s)\cdot\e+\frac12f''(s)\cdot\e^2+o(\e^2)$ for a small $\e$;
\item {\em $C^1$-smooth} if $f$ is differentiable at each point of $U$ and the function $f':U\to X$, $f':u\mapsto f'(u)$, is continuous;
\item {\em absolutely continuous} if for any $\e>0$ there exists $\delta>0$ such that for any points $x_1<y_1<x_2<y_2<\dots<x_n<y_n$ in $U$ with $\sum_{i=1}^n(y_i-x_i)<\delta$ we have $\sum_{i=1}^n\|f(y_n)-f(x_n)\|<\e$;
\item {\em locally absolutely continuous} if for any $s\in U$ there exists a neighborhood $O_s\subseteq U$ of $s$ such that the restriction $f{\restriction}_{O_x}$ is absolutely continuous.
\end{itemize}
For a function $f:U\to X$ we denote by $\dot \Omega_f$ and $\ddot\Omega_f$ the set of points $s\in U$ at which $f$ is differentiable and twice differentiable, respectively.

A subset $A\subseteq \IR$ is called
\begin{itemize}
\item {\em Lebesgue null} if its Lebesgue measure is zero;
\item {\em Lebesgue co-null} if the complement $\IR\setminus A$ is Lebesgue null.
\end{itemize}
By a classical result of Lebesgue \cite[1.2.8]{KK}, the set $\dot\Omega_f$ of differentiability points of any monotone function $f:\IR\to\IR$ is Lebesgue co-null in the real line.

\begin{lemma}\label{l:mac} If a monotone continuous function $f:\IR\to\IR$ is Lipschitz at all but countably many points, then $f$ is locally absolutely continuous.
\end{lemma}

\begin{proof} Let $C$ be the  set of points $s\in\IR$ at which $f$ is not Lipschitz. By Theorem 7.1.38 of \cite{KK}, the local absolute continuity of $f$ will follow as soon as we show that for every bounded Lebesgue null set $E\subseteq \IR$ the image $f(E)$ is Lebesgue null. Let $a=\inf E$ and $b=\sup E$. For every $n\in\IN$ consider the closed subset $$X_n=\bigcap_{y\in[a,b]}\{x\in[a,b]:|f(x)-f(y)|\le n\cdot |x-y|\}$$ of $[a,b]$.
 It follows that $[a,b]\setminus C=\bigcup_{n\in\IN}X_n$. For every $n\in\IN$ the restriction $f{\restriction}_{X_n}$ is a Lipschitz function with Lipschitz constant $n$. Consequently the set $f(X_n\cap E)$ is Lebesgue null. Then the image
 $$f(E)=\bigcup_{x\in E\cap C}\{x\}\cup\bigcup_{n\in\IN}f(E\cap X_n)$$is Lebesgue null being the union of countably many Lebesgue null sets.
 \end{proof}




\subsection{Special directions}

\begin{definition} A point $s\in S_X$ on the unit sphere of a Banach space $X$ is called a {\em special direction} if for any bijective isometry $f:S_X\to S_Y$ onto the unit sphere of an arbitrary Banach space $Y$ and any points $a,b\in S_X$ with $b-a=\|b-a\|\cdot s$ we have $f(b)-f(a)=\|b-a\|\cdot f(s)$.
\end{definition}

In the proof of Theorem~\ref{t:main2} we shall use the following theorem, proved in \cite{BCS}.

\begin{theorem}\label{t:key} A $2$-dimensional Banach space $X$ has the Mazur--Ulam property if its sphere contains two linearly independent special directions.
\end{theorem}



\subsection{Natural parameterizations of spheres}

\begin{definition} Let $X$ be a $2$-dimensional Banach space. A map $\br:\IR\to S_X$ is called a {\em natural parameterization} of the sphere $S_X$ is $\br$ is $C^1$-smooth and $\|\br'(s)\|=1$ for every $s\in\IR$. 
\end{definition}

The following existence and uniqueness theorems for natural parametrizations were proved in \cite{Ban}.

\begin{theorem}\label{t:e} Every smooth $2$-dimensional Banach space $X$ has a natural parameterization $\br:\IR\to S_X$.
\end{theorem}

\begin{theorem}\label{t:u} Let $X,Y$ be two smooth $2$-dimensional Banach spaces and $\br_X:\IR\to S_X$ and $\br_Y:\IR\to S_Y$ be natural parameterizations of their unit spheres. For any isometry $f:S_X\to S_Y$ there exists an isometry $\Phi:\IR\to\IR$ such that $f\circ\br_X=\br_Y\circ\Phi$.
\end{theorem}

Since each isometry $\Phi:\IR\to\IR$ is of the form $\Phi(x)=ax+b$ for some $a,b\in\IR$ with $|a|=1$, Theorems~\ref{t:e}, \ref{t:u} and \ref{t:BCS} imply the following corollary that will be used in the proof of Theorem~\ref{t:main2}.

\begin{corollary}\label{c:nat} If $\br:\IR\to S_X$ is a natural parameterization of the unit sphere of some $2$-dimensional Banach space, then for any bijective isometry $f:S_X\to S_Y$ between $S_X$ and the unit sphere of an arbitrary Banach space $Y$, the map $f\circ\br_X$ is a natural parameterization of $S_Y$.
\end{corollary}

The following lemma proved in \cite[5.4]{Ban} describes a periodicity property of  natural parameterizations.

\begin{lemma}\label{l:r} Let $X$ be a smooth $2$-dimensional Banach space, $\br:\IR\to S_X$ be a natural parameterization of its sphere, and $L=\min\{s\in[0,\infty):\br(L)=-\br(0)\}$. Then $\br(s)=-\br(s+L)=\br(s+2L)$ for every $s\in\IR$.
\end{lemma}

\subsection{Phase shift}

Let $X$ be a smooth $2$-dimensional Banach space and $\br:\IR\to S_X$ be a natural parameterization of its unit sphere. Let $\varphi:\IR\to\IR$ be the function assigning to every $s\in \IR$ the smallest real number such that $\varphi(s)>s$ and $\br'(s)=\br(\varphi(s))$. The function $\varphi$ is called the {\em phase shift} for the parameterization $\br$. Its properties are described in the following lemma taken from \cite[7.1]{Ban}.

\begin{lemma}\label{l:phi} The phase shift $\varphi$ is a continuous non-decreasing function. The Banach space $X$ is strictly convex if and only if the phase shift $\varphi$ is strictly increasing.
\end{lemma}

\begin{lemma}\label{l:difff} If the phase shift $\varphi$ is differentiable at $s\in\IR$, then $\br'$ is differentiable at $s$ and $\br$ is twice differentiable at $s$.
\end{lemma}

\begin{proof} If $\varphi$ is differentiable at $s\in\IR$, then for a small real number $\e$ we have
$$
\begin{aligned}
\br'(s+\e)&=\br(\varphi(s+\e))=\br(\varphi(s)+\varphi'(s)\cdot\e+o(\e))=\\
&=\br(\varphi(s))+\br'(\varphi(s))\cdot(\varphi'(s)\cdot\e+o(\e))+o(\varphi'(s)\cdot\e+o(\e))=\\
&=\br'(s)+\br'(\varphi(s))\cdot\varphi'(s)\cdot\e+o(\e),
\end{aligned}
$$
which means that $\br'$ is differentiable at $s$.

To see that $\br$ is twice differentiable at $s$, observe that for a small real number $\e$ we have
$$
\begin{aligned}
\br(s+\e)-\br(s)&=\int_0^\e\br'(s+t)\,dt=\int_0^\e\br(\varphi(s+t))\,dt=\\
&=\int_0^\e\big(\br(\varphi(s))+\br'(\varphi(s))(\varphi(s+t)-\varphi(s))+o(\varphi(s+t)-\varphi(s))\big)\,dt=\\
&=\br(\varphi(s))\cdot \e+\br'(\varphi(s))\int_0^\e(\varphi'(s)t+o(t))dt+\int_0^\e o(\varphi'(s)t+o(t))dt=\\
&=\br'(s)\cdot \e+\tfrac12\br'(\varphi(s))\cdot\varphi'(s)\cdot\e^2+o(\e^2)
\end{aligned}
$$which means that $\br$ is twice differentiable at $s$.
\end{proof}

By a classical result of Lebesgue \cite[1.2.8]{KK}, the set $\dot\Omega_f$ of differentiability points of any monotone function $f:\IR\to\IR$ is Lebesgue co-null in the real line. This fact and Lemmas~\ref{l:phi}, \ref{l:difff} imply the following lemma.

\begin{lemma}\label{l:full} The sets $\dot\Omega_\varphi\subseteq \dot\Omega_{\br'}\cap\ddot\Omega_\br$ are Lebesgue co-null in the real line.
\end{lemma}

\begin{lemma}\label{l:Las} If $\varphi$ is Lipschitz at all but countably many points, then the Banach space $X$ is absolutely smooth.
\end{lemma}

\begin{proof} By Lemma~\ref{l:mac}, $\varphi$ is locally absolutely continuous and then the function $\br'=\br\circ \varphi$ is locally absolutely continuous being the composition of a $C^1$-smooth function $\br$ and locally absolutely continuous non-decreasing function $\varphi$.
\end{proof}

\subsection{Tingley's Lemma}

We shall need the following lemma proved by Tingley in \cite{Tingley}.

\begin{lemma}[Tingley]\label{l:Tingley} Let $f:S_X\to S_Y$ be an isometry of the unit spheres of finite-dimensional Banach spaces $X,Y$. Then $f(-x)=-f(x)$ for every $x\in S_X$.
\end{lemma}

\section{Some smoothness properties of distances on the sphere}\label{s:dif}

In this section we assume that $X$ is a strictly convex smooth $2$-dimensional Banach space and $\br:\IR\to S_X$ is a natural parameterization of its sphere.
Let $$L=\min\{s\in[0,\infty):\br(s)=-\br(0)\}$$be the half-length of the sphere $S_X$. By Lemma~\ref{l:r}, $\br(s+L)=-\br(s)$ for all $x\in\IR$.

\begin{lemma}\label{l:first} Let $a,b,s\in\IR$ be numbers such that $0\ne\br(b)-\br(a)=\|\br(b)-\br(a)\|\cdot \br(s)$. If the function $\br$ is twice differentiable at $b$ and $s$, then the function $$\nu:\IR\to\IR,\quad\nu:\e\mapsto\|\br(b+\e)-\br(a)\|,$$is twice differentiable at zero.
\end{lemma}

\begin{proof} Let $\br'(b)=x\cdot\br(s)+y\cdot\br'(s)$, $\br''(b)=u\cdot\br(s)+v\cdot\br'(s)$ and $\br''(s)=-\rho\cdot\br(s)+\tau\cdot\br'(s)$ for some real numbers $x,y,u,v,\rho,\tau$.

Given a small $\e$, find a small $\delta$ such that
\begin{equation}\label{eq2}
\br(b+\e)-\br(a)=\|\br(b+\e)-\br(a)\|\cdot\br(s+\delta).
\end{equation}
Taking into account that $\br$ is twice differentiable at $s$ and $b$, we can write
$$\br(s+\delta)=\br(s)+\br'(s)\delta+\tfrac12\br''(s)\delta^2+o(\delta^2)=
(1-\tfrac12\rho\delta^2+o(\delta^2))\br(s)+(1+\tfrac12\tau\delta+o(\delta))\delta\br'(s)$$
and
$$\begin{aligned}
&\br(b+\e)-\br(a)=\br(b)-\br(a)+\br'(b)\e+\tfrac12\br''(b)\e^2+o(\e^2)=\\
&=(\|\br(b)-\br(a)\|+x\e+\tfrac12u\e^2+o(\e^2))\br(s)+(y\e+\tfrac12u\e^2+o(\e^2))\br'(s).
\end{aligned}
$$
Writing the equation (\ref{eq2}) in coordinates, we obtain two equations:
\begin{equation}\label{first}
\|\br(b)-\br(a)\|+x\e+\tfrac12u\e^2+o(\e^2)=\|\br(b+\e)-\br(a)\|\cdot \big(1-\tfrac12\rho\delta^2+o(\delta^2)\big)
\end{equation}
and
\begin{equation}\label{second}
(y+\tfrac12u\e+o(\e))\e=\|\br(b+\e)-\br(a)\|\cdot (1+\tfrac12\tau\delta+o(\delta))\delta.
\end{equation}
The equation (\ref{second}) implies
$$\delta=\frac{y+o(1)}{\|\br(b)-\br(a)\|}\e.$$After substitution of this $\delta$ into the equation (\ref{first}), we obtain
$$
\begin{aligned}
\nu(\e)&=\|\br(b+\e)-\br(a)\|=\frac{\|\br(b)-\br(a)\|+x\e+\tfrac12u\e^2+o(\e)^2}{1-\frac{\rho\cdot(y\e)^2}{2\|\br(b)-\br(a)\|^2}+o(\e^2)}=\\
&=\big(\|\br(b)-\br(a)\|+x\e+\tfrac12u\e^2+o(\e)^2\big)\cdot\big(1+\tfrac{\rho\cdot(y\e)^2}{2\|\br(b)-\br(a)\|^2}+o(\e^2)\big)=\\
&=\|\br(b)-\br(a)\|+x\e+\tfrac12(u+\tfrac{\rho\cdot y^2}{\|\br(b)-\br(a)\|})\e^2+o(\e^2),
\end{aligned}
$$
which means that $\nu$ is twice differentiable at zero.
\end{proof}

\begin{lemma}\label{l:noLip} Let $a,b,s\in\IR$ be numbers such that $0\ne\br(b)-\br(a)=\|\br(b)-\br(a)\|\cdot \br(s)$. If the function $\br$ is twice differentiable at $b$ and the function $$\nu:\IR\to\IR,\quad\nu:\e\mapsto\|\br(b+\e)-\br(a)\|,$$is twice differentiable at zero, then the phase shift $\varphi$ is Lipschitz at $s$.
\end{lemma}

\begin{proof}  Assume that the function $\br$ is twice differentiable at $b$ and the function $\nu$ is twice differentiable at zero.  Let $\br'(b)=x\br(s)+y\br'(s)$ and $\br''(b)=u\br(s)+v\br'(s)$ for some real numbers $x,y,u,v$. 
Since $0\ne\br(b)-\br(a)=\|\br(b)-\br(a)\|\cdot \br(s)$, the strict convexity of $X$ implies that $y\ne0$.

For every $\e\in\IR$ find unique real numbers $\mu(\e),\eta(\e)$ such that $$\br(s+\e)=(1+\mu(\e))\br(s)+(\e+\eta(\e))\br'(s).$$ The differentiability of the function $\br$ at $s$ implies that $\mu(\e)$ and $\eta(\e)$ are of order $o(\e)$ for small $\e$. 

Given a small $\e$, find a small $\delta$ such that
\begin{equation}\label{eq:nu}
\br(b+\delta)-\br(a)=\|\br(b+\delta)-\br(a)\|\cdot\br(s+\e)=\nu(\delta)\cdot\br(s+\e).
\end{equation}
Since the function $\br$ is twice differentiable at $b$, we can write
$$
\begin{aligned}
&\br(b+\delta)-\br(a)=\br(b)-\br(a)+\br'(b)\delta+\tfrac12\br''(b)\delta^2+o(\delta^2)=\\
&=(\|\br(b)-\br(a)\|+x\delta+\tfrac12u\delta^2+o(\delta^2))\br(s)+(y\delta+\tfrac12v\delta^2+o(\delta^2))\br'(s)
\end{aligned}
$$
Since the function $\nu$ is twice differentiable at zero, we obtain
$$\nu(\delta)=\|\br(b)-\br(a)\|+\nu'(0)\delta+\tfrac12\nu''(0)\delta^2+o(\delta^2).$$
Writing the equation (\ref{eq:nu}) in coordinates, we obtain the equations
\begin{equation}\label{eq:nu-one}
\|\br(b)-\br(a)\|+x\delta+\tfrac12u\delta^2+o(\delta^2)=(\|\br(b)-\br(a)\|+\nu'(0)\delta+\tfrac12\nu''(0)\delta^2+o(\delta^2))(1+\mu(\e))
\end{equation}
and
\begin{equation}\label{eq:nu-two}
(y\delta+\tfrac12v\delta^2+o(\delta^2))=(\|\br(b)-\br(a)\|+\nu'(0)\delta+\tfrac12\nu''(0)\delta^2+o(\delta^2))(\e+\eta(\e))
\end{equation}

The equation (\ref{eq:nu-two}) implies
$$\delta=\frac{\|\br(b)-\br(a)\|}{y}\e+o(\e).$$
After substitution of $\delta$ into the  equation (\ref{eq:nu-one}), we obtain
$$
(x-\nu'(0))\delta+\tfrac12(u-\nu''(0))\delta^2+o(\delta^2)=\mu(\e)(\|\br(b)-\br(a)\|+o(1))
$$Taking into account that $\mu(\e)=o(\e)$, we conclude that $x-\nu'(0)=0$ and hence
$$
\begin{aligned}
\mu(\e)&=\frac{(u-\nu''(0))\delta^2+o(\delta^2)}{2\|\br(b)-\br(a)\|+o(1)}=
\frac{(u-\nu''(0))\delta^2}{2\|\br(b)-\br(a)\|}+o(\delta^2)=\\
&=\frac{(u-\nu''(0))\cdot\|\br(b)-\br(a)\|}{2y^2}\e^2+o(\e^2),\end{aligned}
$$
which means that $\mu$ is twice differentiable at zero and $$\mu''(0)=\frac{(u-\nu''(0))\cdot\|\br(b)-\br(a)\|}{y^2}.$$
 
Write the vector $\br'(\varphi(s))$ in the basis $\br(s),\br'(s)$ as 
$$\br'(\varphi(s))=-P(s)\cdot\br(s)+T(s)\cdot\br'(s)$$for some real numbers $P(s),T(s)$ (called the {\em radial} and {\em tangential supercuravtures} at $s$, see \cite[\S7]{Ban}). Since the vector $\br'(\varphi(s))$ is not collinear to $\br(\varphi(s))=\br'(s)$, the radial supercurvature $P(s)$ is not equal to zero. 
 
Since the function $\br$ is $C^1$-smooth, for small $\e$ we have the equality 
$$
\begin{aligned}
&\mu(\e)\br(s)+(\e+\eta(\e)\br'(s)=\br(s+\e)-\br(s)=\int_0^\e\br'(s+u)\,du=\int_0^\e\br(\varphi(s+u))\,du=\\
&=\br(\varphi(s))\cdot\e+\int_0^\e(\br(\varphi(s+u))-\br(\varphi(s))\,du=\br'(s)\cdot\e+\int_0^\e\int_{\varphi(s)}^{\varphi(s+u)}\br'(t)\,dt\,du=\\
&=\br'(s)\cdot\e+\int_0^\e\int_{\varphi(s)}^{\varphi(s+u)}\big(\br'(\varphi(s))+o(1)\big)\,dt\,du=\\
&=\br'(s)\cdot\e+\big(\br'(\varphi(s))+o(1)\big)\int_0^\e(\varphi(s+u)-\varphi(s))\,du=\\
&=\br'(s)\cdot\e+(-P(s)\br(s)+T(s)\br'(s)+o(1))\int_0^\e(\varphi(s+u)-\varphi(s))du
\end{aligned}
$$
and hence 
$$\tfrac12\mu''(0)\e^2+o(\e^2)=\mu(\e)=(-P(s)+o(1))\int_0^\e(\varphi(s+u)-\varphi(s))du.$$
Thaking into account that the function $\varphi$ is monotone, we obtain
$$\big|\varphi(s+\e)-\varphi(s)\big|\cdot|\e|\le\Big|\int_\e^{2\e}(\varphi(s+u)-\varphi(s))du\Big|\le\Big|\int_0^{2\e}(\varphi(s+u)-\varphi(s))du\Big|=\Big|\frac{\mu''(0)+o(1)}{2\cdot P(s)}\Big|(2\e)^2,$$
which means that $\varphi$ is Lipschitz at $s$.
\end{proof}

\begin{lemma}\label{l:spec2} If the function $\varphi$ is not Lipschitz at a point $s\in\IR$, then the direction $\br(s)\in S_X$ is special.
\end{lemma}

\begin{proof} Let $f:S_X\to S_Y$ be a bijective isometry of $S_X$ onto the unit sphere $S_Y$ of an arbitrary Banach space $Y$. It is clear that the Banach space $Y$ is $2$-dimensional. Theorems~\ref{t:CS} and \ref{t:BCS} imply that the Banach space $Y$ is strictly convex and smooth. By Corollary~\ref{c:nat}, the composition $\br_Y=f\circ\br:\IR\to S_Y$ is a natural parameterization of the sphere $S_Y$.

Let $\theta:S_X\to S_X$ be the map assigning to each $x\in S_X$ the unique point $\theta(x)\in S_X$ such that $\{x,\theta(x)\}=S_X\cap(x+\IR\cdot\br(s))$. The strict convexity of $X$ implies that the map $\theta$ is well-defined and continuous. Since $\theta\circ\theta$ is the identity map of $S_X$, the map $\theta$ is a homeomorphism of the sphere $S_X$. The definition of $\theta$ ensures that  $x-\theta(x)\in\{\|x-\theta(x)\|\cdot \br(s),-\|x-\theta(x)\|\cdot\br(s)\}$ for every $x\in S_X$. 

Let $$P=\{x\in S_X:0\ne x-\theta(x)=\|x-\theta(x)\|\cdot \br(s)\}.$$ Then $S_X=P\cup\{x\in S_X:\theta(x)=x\}\cup(-P)$. Observe that $\theta(\br(s))=-\br(s)$ and hence $\br(s)\in P$.

Find a real number $a\in \IR$ such that $\{\br(a),\br(a+L)\}=\{x\in S_X:\theta(x)=x\}$, $\br\big((a,a+L)\big)=P$,  and $a<s<a+L$. Then $\br{\restriction}_{(a-L,a)}$ is a homeomorphism of the open interval $(a-L,a)$ onto the open half-sphere $-P$. 

The $C^1$-smoothness of the function $\br$ implies the $C^1$-smoothness of the function 
$$\Theta=(\br{\restriction}_{(a-L,a)})^{-1}\circ \theta\circ \br{\restriction}_{(a,a+L)}:(a,a+L)\to(a,a-L).$$ 
Since $\theta(\br(s))=-\br(s)=\br(s-L)$, we have $\Theta(s)=s-L$.

Now consider the continuous map
$$\phi:(a,a+L)\to Y,\;\phi:t\mapsto {f\circ \br(t)-f\circ \theta\circ \br(t)}.$$

\begin{claim} The map $\phi$ is $C^1$-smooth.
\end{claim}

\begin{proof} The function $\br_Y=f\circ \br$ is $C^1$-smooth, being a natural parameterization of the sphere $S_Y$ of the smooth Banach space $Y$.

 Now take any $t\in(a,a+L)$ and observe that $\br(t)\in P$, $\theta(\br(t))\in -P$ and  hence $\Theta(t)=(\br{\restriction}_{(a-L,a)})^{-1}\circ \theta\circ\br(t)$ is well-defined.
Since 
$$f\circ \theta\circ \br(t)=f\circ \br\circ(\br{\restriction}_{(a-L,a)})^{-1}\circ \theta\circ\br(t)=\br_Y\circ\Theta(t),$$the function $f\circ\theta\circ\br$ is is continuously differentiable at $t$ (by the $C^1$-smoothness of the functions $\Theta$ and $\br_Y$).
\end{proof}

The smoothness of the sphere $S_Y$ implies the continuous differentiability of the norm $\|\cdot\|:Y\to\IR$ on the set $Y\setminus \{0\}$. Then the function
$$\Phi:(a,a+L)\to S_Y,\quad \Phi:t\mapsto\frac{\phi(t)}{\|\phi(t)\|},$$
is $C^1$-smooth. By Lemma~\ref{l:Tingley}, $$\phi(s)=f(\br(s))-f(\theta(\br(s)))=f(\br(s))-f(-\br(s))=f(\br(s))-(-f(\br(s)))=2\cdot f(\br(s))$$ and hence $\Phi(s)=f(\br(s))$. Assuming that the function $\Phi$ is not constant, we can find a point $c\in (a,a+L)$ such that $\Phi$ is continuously differentiable at $c$ and the derivative $\Phi'(c)$ is not zero. Then $\Phi$ is a diffeomorphism of some neighborhood $U_c\subseteq (a,a+L)$ onto its image $\Phi(U_c)\subseteq S_Y$.

By Lemma~\ref{l:full}, the sets $\ddot\Omega_\br$ and $\ddot\Omega_{\br_Y}$ are Lebesgue co-null. Since diffeomorphisms preserve Lebesgue co-null sets, the set $U_c\cap \ddot\Omega_{\br}\cap\ddot\Omega_{\br_Y}\cap\Phi^{-1}(\br_Y(\ddot\Omega_{\br_Y}))$ is Lebesgue co-null in $U_c$ and hence contains some point $b$.
Let $a=\Theta(b)$ and observe that $\br(b)-\br(a)=\br(b)-\theta(\br(b))=\|\br(b)-\br(a)\|\cdot\br(s)$. Since $\Phi(b)\in \br_Y(\ddot\Omega_{\br_Y})$, there exists a real number $y\in \ddot\Omega_{\br_Y}$ such that $\Phi(b)=\br_Y(y)$.

Since $b\in\ddot\Omega_\br$ and $\varphi$ is not Lipschitz at $s$, the function $$\nu:\IR\to\IR,\quad \nu:\e\mapsto \|\br(b+\e)-\br(a)\|$$ is not twice differentiable at zero according to Lemma~\ref{l:noLip}.

On the other hand, we have
\begin{multline*}
\br_Y(b)-\br_Y(a)=f(\br(b))-f(\br(a))=f(\br(b))-f(\theta(\br(b)))=\\
=\Phi(b)\cdot\|f(\br(b))-f(\theta(\br(b))\|=\br_Y(y)\cdot \|\br_Y(b)-\br_Y(a)\|
\end{multline*}
and for every $\e$
$$\nu(\e)=\|\br(b+\e)-\br(a)\|=\|f\circ\br(b+\e)-f\circ\br(a)\|=\|\br_Y(b+\e)-\br_Y(a)\|.$$
Since $b,y\in\ddot\Omega_{\br_Y}$, the function
$$\nu:\IR\to\IR,\quad\nu_Y:\e\mapsto\|\br_Y(b+s)-\br_Y(a)\|=\|\br(b+\e)-\br(a)\|,$$
is twice differentiable at zero by Lemma~\ref{l:first}.
 This contradiction shows that the function $\Phi:(a,a+L)\to S_Y$ is constant with $f(\br(s))\in \Phi\big((a,a+L)\big)=\{f(\br(s))\}$ and hence the direction $\br(s)\in S_X$ is special.
\end{proof}

\section{Proof of Theorem~\ref{t:main2}}\label{s:main}

Assume that a $2$-dimensional Banach space $X$ is strictly convex, smooth, and not absolutely smooth. Let $\br:\IR\to S_X$ be a natural parameterization of the unit sphere $S_X$ and $\varphi:\IR\to\IR$ is the phase shift. Let $C$ be the set of points $s\in\IR$ at which the function $\varphi$ is not Lipschitz. Since $X$ is not absolutely smooth, the set $C$ is uncountable according to Lemma~\ref{l:Las}. Then we can choose two points $a,b\in C$ such that the vectors $\br(a), \br(b)$ are linearly independent. By Lemma~\ref{l:spec2}, the directions $\br(a),\br(b)$ are special and by Theorem~\ref{t:key}, the Banach space $X$ has the Mazur--Ulam property.

\section{Acknowledgements} 

The author expresses his sincere thanks to Javier Cabello S\'anchez for his inspiring paper \cite{CS} that contained a crucial idea of special directions (appearing  explicitly in \cite{BCS}), which allowed to handle non-(absolutely)-smooth cases in the proof of Theorem~\ref{t:main}.

\end{document}